\documentclass[12pt, fleqn]{article}
\usepackage[german,english]{babel}
\usepackage[latin1]{inputenc}
\usepackage{times,amsmath,amssymb,graphicx}
\usepackage{theorem,bm}

{\theorembodyfont{\rm}
  \newtheorem{defi}{Definition}[section]
  \newtheorem{prob}[defi]{Problem}

}
  
  \newtheorem{prop}[defi]{Proposition}
  \newtheorem{thm}[defi]{Theorem}%[defi]

\newcommand{\para}{\vspace{3pt plus 2pt minus 1pt}
\refstepcounter{defi}\textbf{\arabic{section}.\arabic{defi}}\, }

\newcommand{\vx}{{\bm x}}
\newcommand{\PP}{{\mathbb P}}
\newcommand{\RR}{{\mathbb R}}
\newcommand{\T}{{\mathrm{T}}}
\newcommand{\GL}{{\mathrm{GL}}}

\newcommand{\abc}{{\alpha,\beta,\gamma}}
\newcommand{\alphaq}{{\overline{\vphantom{\beta}\alpha}}}
\newcommand{\betaq}{{\overline\beta}}
\newcommand{\gammaq}{{\overline{\vphantom{\beta}\gamma}}}
\newcommand{\abcq}{{\,\alphaq,\betaq,\gammaq\,}}

\let\phi=\varphi

\let\theta=\vartheta
\newcommand{\Matrixfeld}[4]{\left#1\!\begin{array}{*{#3}{c}}#4\end{array}\!\right#2}

\newcommand{\Mat}{\Matrixfeld()}

\newenvironment{proof}%[1]%
    {\begin{trivlist} \item {\emph{Proof}.}} %#1\/:}}%
    {{}\hfill $\square$ \end{trivlist}} %%Hans {}\hfill erg{\"a}nzt und \/ gel{\"o}scht

\newcounter{abbildung} %% Zaehler fuer Abbildungen
\begin{document}
%\selectlanguage{german}
%-------------------------------------------
% if you write in English, please uncomment
% the following line i.e. remove the % -sign
%-------------------------------------------
\selectlanguage{english}

\vspace*{0.7cm} \centerline{\LARGE\bf
%-------------------------------------------
% Bitte in die nachstehende Leerzeile den TITEL
% Ihres Beitrages eingeben /
% Please enter the TITLE in the following
% blank line
%-------------------------------------------
HIGHER ORDER CONTACT}

\vspace*{0.7cm} \centerline{\LARGE\bf ON CAYLEY'S RULED CUBIC SURFACE}
 \vspace*{1.4cm}

\centerline{\large Hans Havlicek (Wien)
%-------------------------------------------
% Bitte in die nachstehende Leerzeile den
% AUTORENNAMEN und den ORT eintragen /
% Please enter the AUTHOR'S NAME and LOCATION
% in the following blank line
% z.B. / e.g.
% Max Mustermann (Dresden)
%-------------------------------------------
} \vspace*{0.6cm}

{\small
%-------------------------------------------
% Bitte nachstehend den ABSTRACT eintragen /
% Please enter the ABSTRACT after the next
% line
%-------------------------------------------
\emph{Abstract}: It is well known that Cayley's ruled cubic surface carries a
three-parameter family of twisted cubics sharing a common point, with the same
tangent and the same osculating plane. We report on various results and open
problems with respect to contact of higher order and dual contact of higher
order for these curves.
 }

 {\small
 \emph{Keywords}: Cayley surface, twisted cubic, contact of higher order, dual
contact of higher order, twofold isotropic space.

}
%-------------------------------------------
% Bitte nachstehend den TEXT eingeben /
% Please enter the main TEXT after this
% comment
%
%-------------------------------------------
%
% BITTE BEACHTEN:
% Weitere UEBERSCHRIFTEN mit den Befehlen \section{..}
% und \subsection{..} bilden,
% das \section zum Eintragen der Ueberschrift ist unten bereits
% vorgegeben
%
% Abgesetzte GLEICHUNGEN z.B. mit
% \vspace{-0.8cm}
% \begin{equation} ... \end{equation}
% \vspace{-0.3cm}
% eingeben,
% damit eine Nummerierung erfolgt, falls erwuenscht
%
% Fuer LITERATURANGABEN am Ende des Textes bitte
% \begin{thebibliography}{00}...\end{thebibliography}
% \bibitem{...} Autor
% benutzen
%-------------------------------------------

\section{Introduction}\label{se:introduction}

\para
The author's interest in higher order contact of twisted cubics and in Cayley's
ruled cubic surface arose some time ago when investigating a three-dimensional
analogue of Laguerre's geometry of spears in terms of higher dual numbers. We
refer to our joint paper with Klaus List \cite{havl+l-03a} for further details.
\par

It is well known that Cayley's ruled cubic surface $F$ or, for short, the
\emph{Cayley surface\/}, carries a three-parameter family of twisted cubics
$c_\abc$; cf.\ formula (\ref{eq:abc_proj}) below. All of them share a common
point $U$, with the same tangent $t$ and the same osculating plane $\omega$,
say. Thus all these curves touch each other at $U$; some of them even have
higher order contact at $U$.
\par

Among the curves $c_\abc$ are the \emph{asymptotic curves\/} of $F$. They form
a distinguished subfamily. When speaking here of asymptotic curves of $F$ we
always mean asymptotic curves other than generators. In a paper by Hans
Neudorfer, which appeared in the year 1925, it is written that the osculating
curves have contact of \emph{order four\/} at $U$. This statement can also be
found elsewhere, e.g.\ in \cite[p.~232]{mueller+k3-31}. Neudorfer did not give
a formal proof. He emphasized instead that his statement would be immediate,
since under a projection with centre $U$ (onto some plane) the images of the
asymptotic curves give rise to a pencil of hyperosculating conics
\cite[p.~209]{neud-25a}. Ten years had to pass by before Walter Wunderlich
determined a correct result in his very first publication
\cite[p.~114]{wund-35a}. He showed that distinct asymptotic curves have contact
of \emph{order three\/} at $U$, even though their projections through the
centre $U$ have order four, as was correctly noted by Neudorfer. The reason for
this discrepancy is that the centre of projection coincides with the point of
contact of the asymptotic curves.
\par
The inner geometry of Cayley's surface was investigated thoroughly by Heinrich
Brauner in \cite{brau-64}. He claimed in \cite[pp.~96--97]{brau-64} that two
twisted cubics of our family $c_\abc$ with contact of order four are identical.
However, this contradicts a result in \cite[p.~126]{havl+l-03a}, where a
one-parameter subfamily of twisted cubics with this property was shown to
exist.
\par
The author accomplished the task of describing the order of contact between the
twisted cubics $c_\abc$ in \cite{havl-04z}. The proof consists of calculations
which are elementary, but at the same time long and tedious, whence a computer
algebra system (Maple) was used. This could explain why we could not find such
a description in one of the many papers on the Cayley surface.

\par\para
The aim of the present note is to say a little bit more about the results
obtained in \cite{havl-04z}, where we did not only characterize contact of
higher order at $U$, but also the order of contact of the associated dual
curves (\emph{cubic envelopes\/}) at their common plane $\omega$. Furthermore,
we state open problems, since for some of those results no geometric
interpretation seems to be known.

\par
There is a wealth of literature on the Cayley surface and its many fascinating
properties. We refer to \cite{brau-67c}, \cite{dill+k-99a}, \cite{gmai+h-04z},
\cite{havl-04z}, \cite{husty-84}, \cite{koch-68}, \cite{nomizu+s-94a},
\cite{oehler-69}, and \cite{wiman-36}. This list is far from being complete,
and we encourage the reader to take a look at the references given in the cited
papers.

\section{The Cayley surface}\label{se:cayley}

\para
In this note we consider the three-dimensional real projective space and denote
it by $\PP_3(\RR)$. A point is of the form $\RR\vx$ with
$\vx=(x_0,x_1,x_2,x_3)^\T \in \RR^{4\times 1}$ being a non-zero column vector.
We consider the plane $\omega$ with equation $x_0=0$ as the \emph{plane at
infinity}, and we regard $\PP_3(\RR)$ as a projectively closed affine space. We
shall use some notions from projective differential geometry without further
reference. All this can be found in \cite{bol-50} and \cite{degen-94}.

\par\para The following is taken from \cite{havl-04z}, where we followed
\cite{brau-64}. \emph{Cayley's} (\emph{ruled cubic\/}) \emph{surface} is, to
within collineations of $\PP_3(\RR)$, the surface $F$ with equation
\begin{equation}\label{eq:cayley}
  3x_0x_1x_2-x_1^3-3x_3x_0^2=0.
\end{equation}
The Cayley surface contains the line $t:x_0=x_1=0$, which is a torsal generator
of second order and at the same time a directrix for all other generators of
$F$. The point $U=\RR(0,0,0,1)^\T$ is the cuspidal point on $t$; it is a
so-called \emph{pinch point} \cite[p.~181]{mueller+k3-31}. No generator of $F$
other than $t$ passes through $U$. Each point of $t\setminus\{U\}$ is incident
with precisely one generator $\neq t$. Likewise, each plane through $t$ other
than $\omega$ intersects $F$ residually along a generator $\neq t$. The plane
$\omega$ meets $F$ at $t$ only. See Figure \ref{abb1}.
\par
The set of all matrices
\begin{equation}\label{eq:M.abc}
  M_{a,b,c}:= \Mat4{
  1&0&0&0\\
  a&c&0&0\\
  b&ac&{c}^{2}&0\\
  ab-\frac{1}{3}{a}^{3}&bc&a{c}^{2}&{c}^{3}
  }
\end{equation}
where $a,b \in \RR$ and $c \in \RR\setminus\{0\}$ is a three-parameter Lie
group, say $G$, under multiplication. Confer \cite[p.~96]{brau-64}, formula
(9). All one- and two-parameter subgroups of $G$ were determined in
\cite[pp.~97--101]{brau-64}.
\par
The group $G$ acts faithfully on $\PP_3(\RR)$ as a group of collineations
fixing $F$. Under the action of $G$, the points of $F$ fall into three orbits:
$F\setminus\omega$, $t\setminus\{U\}$, and $\{U \}$. The group $G$ yields
\emph{all} collineations of $F$; cf.\ \cite[Section~3]{gmai+h-04z}, where this
problem was addressed in the wider context of an arbitrary ground field. Since
we did not exclude fields of characteristic $3$ in that article, it was
necessary there to rewrite equation (\ref{eq:cayley}) in a slightly different
form.
%%%%%%%%%%%%%%%%%%%%%%%%%%%%%%%%%%%%%%%%%%%%%%%%%%%%%%%%%%%%%%%%%%
{\unitlength1.0cm %% Figur mit Kuspidalpunkt
      \begin{center}\small
      \begin{picture}(6.88,5.82)
         \put(0.0 ,0.0){\includegraphics[width=6.88\unitlength]{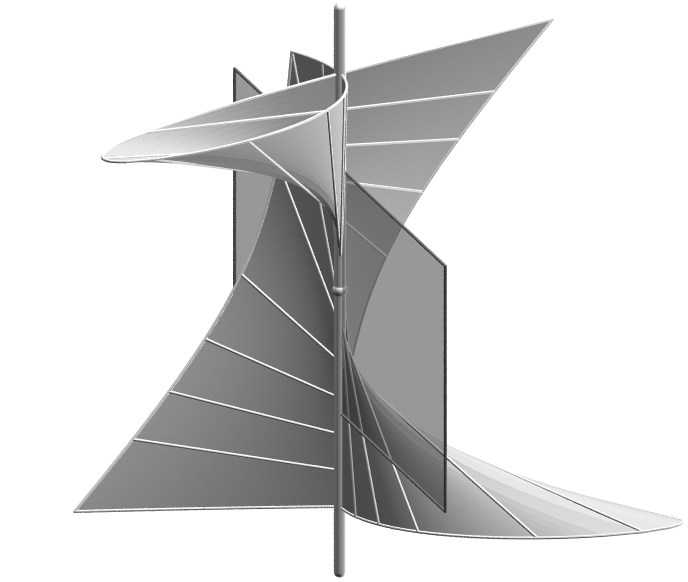}}
%%         \put(0.0 ,0.0){\includegraphics[width=7.0\unitlength]{./bilder/gitter.eps}}
         \put(3.1,5.5){$t$}
         \put(3.45,2.9){$U$}
         \put(4.1,3.0){$\omega$}
         \put(4.6,5.05){$F$}
      \end{picture}\\
      {\refstepcounter{abbildung}\label{abb1}
      {Figure \ref{abb1}}}
\end{center}
}%
%%%%%%%%%%%%%%%%%%%%%%%%%%%%%%%%%%%%%%%%%%%%%%%%%%%%%%%%%%%%%%%%%%

\section{A family of cubic parabolas}

\para We now turn to the family of twisted cubics on $F$ already mentioned in the
introduction. From an affine point of view all these twisted cubics are
\emph{cubic parabolas}, i.e., the plane at infinity is one of their osculating
planes. For the sake of completeness we state the next result together with its
short proof:

\begin{prop}\label{prop:kubiken}
Let $\alpha, \beta, \gamma \in \RR$ such that $\beta \ne 0, 3$. Then
\begin{equation}\label{eq:abc_proj}%% LAYOUT negative Abst{\"a}nde
c_{\alpha, \beta, \gamma}\! :=\! \left\{\!\RR\big(1, u-\gamma,
 \frac{u^2+\alpha}{\beta},
 \frac{u-\gamma}{3\beta}\big(3(u^2+\alpha) - \beta(u-\gamma)^2\big)\big)^\T
 \!\mid\! u\in\RR\!\cup\{\infty\}\!
 \right\}
\end{equation}
is a cubic parabola on the Cayley surface $F$. Conversely, the set formed by
all such $c_{\alpha, \beta, \gamma}$ coincides with the set of all those
twisted cubics on $F$ which contain $U$, have $t$ as a tangent line, and
$\omega$ as an osculating plane.
\end{prop}
\begin{proof}
The first assertion is immediately seen to be true by a straightforward
calculation.

In order to show the converse, we consider an arbitrary twisted cubic $c'
\subset F$ passing through $U$, touching the line $t$, and osculating the plane
$\omega$. So $c'$ is a cubic parabola. Our aim is to show that $c'=c_\abc$ for
certain elements $\abc\in \RR$ with $\beta\neq 0,3$. The auxiliary cubic
parabola
\begin{equation*}
  c= \{\RR(1,u,u^2,u^3)^\T  \mid u \in \RR \cup \{\infty\} \}
\end{equation*}
is not on $F$. By a well known result (see, e.g.\ \cite[p.~204]{brau-76-2}),
for any two triads of distinct points on $c$ and $c'$ there exists a unique
collineation of $\PP_3(\RR)$ which takes the first to the second triad and the
twisted cubic $c$ to $c'$.
\par
The cubic parabolas $c$ and $c'$ have the common point $U$, with the same
tangent $t$ and the same osculating plane $\omega$. Every plane $\pi\supset t$
with $\pi \ne \omega$ intersects $c$ and $c'$ residually at a unique point
$\neq U$. By the above, there exists a unique collineation $\kappa$, say,
taking $c$ to $c'$ such that $U$ remains fixed and such that
$\RR(1,0,0,0)^\T\in c$ and $\RR(1,1,1,1)^\T\in c$ go over to affine points of
$c'$ in the planes $x_1=0$ and $x_0-x_1=0$, respectively. As $U$ is fixed, so
are the tangent line $t$ and the osculating plane $\omega$. Hence $\kappa$ is
an affinity. Altogether, $\kappa$ fixes three distinct planes through $t$.
Consequently, all planes of the pencil with axis $t$ remain invariant. Thus
$\kappa$ can be described by a lower triangular matrix of the form
\begin{equation}\label{}
A := \Mat4{
  1 &  &  &  \\
  0 & 1 &  &  \\
  a_{20} & a_{21} & a_{22} &  \\
  a_{30} & a_{31} & a_{32} & a_{33}  } \ \in \GL_4(\RR).
\end{equation}
By substituting $A \cdot (1,u,u^2,u^3)^{\T}$ in the left hand side of
(\ref{eq:cayley})  we get the polynomial identity
\begin{equation}
 3\left(\big(a_{22}-a_{33}-\frac13\big)u^3+ (a_{21}-a_{32})u^2
 + (a_{20}-a_{31})u - a_{30}\right)=0 \
\mbox{ for all } \ u \in \RR.
\end{equation}
The coefficients at the powers of $u$ have to vanish. So we obtain
\begin{equation}\label{eq:bedingungen}
  a_{30}=0,\; a_{31}=a_{20},\; a_{32}=a_{21},
  \mbox{ and } a_{33}=a_{22}-\frac13=0.
\end{equation}
Also, since $A$ is regular, $a_{22} \ne 0,\frac13$. Taking into account
(\ref{eq:bedingungen}) the column vector $A \cdot (1,u,u^2,u^3)^{\T}$ turns
into
\begin{equation}\label{}
    \left(1,u,a_{20}+ a_{21}u+ a_{22}{u}^{2},a_{20}u+a_{21}{u}^{2}+
    \big(a_{22}-\frac13\big){u}^{3}\right)^\T.
\end{equation}
Then, putting $\beta:=1/a_{22}$, $\gamma:=a_{21}b/2$, $\alpha:=a_{20}b-c^2$,
and $u':=u-c$ shows that $c'=c_{\alpha, \beta, \gamma}$, as required.
\end{proof}

\par\para
If we allow the exceptional value $\beta=3$ in (\ref{eq:abc_proj}) then the
vector in that formula simplifies to
\begin{equation}\label{}
    \big( 1, u-\gamma, \frac{u^2+\alpha}{3},
    \frac{u-\gamma}{3}\big( \alpha+2u\gamma-\gamma^{2}    \big)\big)^\T.
\end{equation}
In fact, we get in this particular case a parabola, say $c_{\alpha,3,\gamma}$,
lying on $F$. Each parabola of this kind is part of the intersection of the
Cayley surface $F$ with one of its tangent planes. Clearly, we cannot have
$\beta=0$ in (\ref{eq:abc_proj}).
\par
Each curve $c_\abc$ ($\beta\neq 0$) is on the \emph{parabolic cylinder\/} with
equation
\begin{equation}\label{eq:zylinder}
    \alpha x_0^2-\beta x_0x_2+(x_1+\gamma x_0)^2 = 0.
\end{equation}
In projective terms the vertex of this cylinder is the point $U$. (The
intersection of this cylinder with the plane $x_3=0$ gives the projection of
$c_\abc$ used by Neudorfer in \cite{neud-25a}.) The mapping $( \abc ) \mapsto
c_\abc$ is injective, since different triads $(\abc)\in\RR^3$, $\beta\neq 0$,
yield different parabolic cylinders (\ref{eq:zylinder}).
\par
It should be noted here that Proposition~\ref{prop:kubiken} does \emph{not}
describe all twisted cubics on $F$. There are also twisted cubics on $F$ with
two distinct points at infinity. We shall not need this result, whence we
refrain from a further discussion.

\par\para
The the group $G$ of all matrices (\ref{eq:M.abc}) acts on the family of all
cubic parabolas $c_\abc$. Each matrix $M_{a,b,c}\in G$ takes a cubic parabola
$c_\abc$ to a cubic parabola, say $c_{\alphaq,\betaq,\gammaq}$. The values
$\abcq$ are given as follows:
\begin{eqnarray}\label{eq:alpha}
   \alphaq &=& -a^2\,\frac{\beta^2}4-a c\beta\gamma+c^2\alpha+b\beta, \\
  \betaq &=& \beta,\label{eq:beta} \\
  \gammaq &=& a\,\frac{\beta-2}{2}+c\gamma.\label{eq:gamma}
\end{eqnarray}
See \cite[p.~97]{brau-64}, formula (12), where some signs in the formula for
$\alphaq$ are reproduced incorrectly.

\par\para\label{para:kennzeichen}
An interpretation of the three numbers $\abc$, associated with each of our
cubic parabolas, can be given in terms of a \emph{distance function\/} on
$F\setminus t$ which is due to Brauner. This distance function assumes all real
numbers and the value $\infty$; it does not satisfy the triangle inequality.
Moreover, the distance from a point $X$ to a point $Y$ is in general
\emph{not\/} the distance from $Y$ to $X$. The distance function is a
$G$-invariant notion. See \cite[pp.~115]{brau-64} for further details.

By (\ref{eq:beta}), the parameter $\beta$ is a $G$-invariant notion, whereas
$\alpha$ and $\gamma$ describe, loosely speaking, the ``position'' of the cubic
parabola on the Cayley surface. To be more precise, up to one exceptional case
the following holds: Each curve $c_\abc$ contains a \emph{circle\/} with
midpoint $M$ and radius $\rho\neq 1$ in the sense of Brauner's distance
function. Due to the specific properties of this distance function, the
midpoint $M$ is also on $c_\abc\setminus\{U\}$. (The tangent to $c_\abc$ at a
midpoint is an osculating tangent.) The parameter $\beta$ equals
$2\rho/(\rho-1)$, whereas $\alpha$ and $\gamma$ can be expressed in terms of
the midpoint $M$ and the radius $\rho$. All this can be read off from formula
(70) in \cite[pp.~116]{brau-64}.

Only the curves $c_\abc$ with $\beta=2$ and $\gamma\neq 0$ do not allow for
this interpretation, since none of their tangents at an affine point is an
osculating tangent of $F$. On the other hand, the asymptotic curves of $F$ are
obtained for any $\alpha\in\RR$, $\beta=2$, and $\gamma=0$, whence
$\rho=\infty$. They play a special role: \emph{Each affine point of an
asymptotic curve can be considered as a midpoint}.
\par
In \cite[Section~5]{gmai+h-04z} an alternative approach to Brauner's distance
function was given in terms of cross-ratios. As a matter of fact, the distance
functions in \cite{brau-64} and \cite{gmai+h-04z} are identical up to a
bijection of $\RR\cup\{\infty\}$. This difference became necessary in order to
establish our results in \cite{gmai+h-04z} for a wider class of projective
spaces. \emph{Each distance preserving mapping $\phi:F\setminus t\to F\setminus
t$ comes from a unique matrix in $G$} \cite[Theorem~5.5]{gmai+h-04z}. No
assumptions like injectivity, surjectivity or even differentiability of $\phi$
are needed for the proof. This means that Brauner's distance function is a
\emph{defining function\/} (see \cite[p.~23]{benz-94}) for the Lie group $G$.

\par\para
We now turn to our original problem of describing the order of contact at $U$
of our cubic parabolas given by (\ref{eq:abc_proj}). Also, we state necessary
and sufficient conditions for their \emph{dual curves} (formed by all
osculating planes) to have contact of a prescribed order at $\omega$. We
shortly speak of \emph{dual contact\/} in this context. Recall that the
distinction between contact and dual contact cannot be avoided in three
dimensional projective differential geometry, whereas for planar curves the
concepts of contact and dual contact are self-dual notions. We refer also to
\cite[Theorem~1]{degen-88} for explicit formulas describing contact of higher
order between curves in $d$-dimensional real projective space.

Since twisted cubics with (dual) contact of order five are identical
\cite[pp.~147--148]{bol-50}, we may restrict our attention to distinct curves
with (dual) contact of order less or equal four. The following was shown in
\cite[Theorem~1]{havl-04z} and \cite[Theorem~3]{havl-04z}:

\begin{thm}\label{thm:1}
Distinct cubic parabolas $c_\abc$ and $c_\abcq$ on Cayley's ruled surface have
\vspace{-2ex}
\begin{enumerate}\itemsep0ex
\item second order contact at $U$ if, and only if, $\beta=\betaq$ or
$\beta=3-\betaq$;

\item third order contact at $U$ if, and only if, $\beta=\betaq$ and
$\gamma=\gammaq$, or  $\beta=\betaq=\frac32$;

\item fourth order contact at $U$ if, and only if,
$\beta=\betaq=\frac32$ and $\gamma=\gammaq$.

\item second order dual contact at $\omega$ if, and only if, $\beta=\betaq$;

\item third order dual contact at $\omega$ if, and only if, $\beta=\betaq$ and
$\gamma=\gammaq$, or  $\beta=\betaq=\frac52$;

\item fourth order dual contact at $\omega$ if, and only if,
$\beta=\betaq=\frac73$ and $\gamma=\gammaq$.

\end{enumerate}
\end{thm}

It follows from Theorem \ref{thm:1} that cubic parabolas $c_\abc$ with
\begin{equation}\label{eq:sonderwerte}
  \beta=\frac32,\; \beta=\frac52, \mbox{ and }\beta=\frac73
\end{equation}
should play a special role. Note that none of them yields the asymptotic curves
of $F$, as they have the form $c_\abc$ with $\alpha\in\RR$, $\beta=2$, and
$\gamma=0$.

For some of the values in (\ref{eq:sonderwerte}) we could find a geometric
interpretation, but a lot of open problems remain.

\par\para
The flag $(U,t,\omega)$ turns $\PP_3(\RR)$ into a \emph{twofold isotropic} (or
\emph{flag}) \emph{space}. The definition of metric notions (distance, angle)
in this space is based upon the identification of
$\RR(1,x_1,x_2,x_3)^\T\in\PP_3(\RR)\setminus\omega$ with
$(x_1,x_2,x_3)^\T\in\RR^{3\times 1}$, and the canonical basis of $\RR^{3\times
1}$, which defines the units for all kinds of measurements in this space. See
\cite{brau-66} for a detailed description.
\par
By \cite[p.~137]{brau-67a}, each cubic parabola $c_\abc$ has
\begin{equation}\label{}
  \frac12\,\beta (3-\beta)\leq\frac98
\end{equation}
as its \emph{conical curvature\/} in the sense of the twofold isotropic space.
\emph{Among all cubic parabolas $c_\abc$, the ones with $\beta=\frac32$ are
precisely those whose conical curvature attains the maximal value $\frac98$}
\cite[Theorem~2]{havl-04z}. Of course, this is a characterization in terms of
the \emph{ambient space\/} of $F$, whence we are lead to the following
question:

\begin{prob}
Is there a characterization of the cubic parabolas $c_\abc$ with
$\beta=\frac32$ in terms of the \emph{inner geometry} on the Cayley surface?
\end{prob}
By \emph{inner geometry\/} of $F$ we mean here the geometry on $F\setminus t$
given by the action of the group $G$, in the sense of Felix Klein's Erlangen
programme.
\par\para
The next noteworthy value is $\beta=\frac52$. Presently, its meaning is not at
all understood by the author. So we can merely state the following:

\begin{prob}
Give \emph{any\/} geometric characterization of the cubic parabolas $c_\abc$
with $\beta=\frac52$.
\end{prob}

\par\para
Finally, we turn to $\beta=\frac73$. If we take two distinct cubic parabolas
$c_\abc$ and $c_\abcq$ with parameters
\begin{equation}\label{}
  \beta=\betaq=\frac73 \mbox{ and } \gamma=\gammaq=0
\end{equation}
then their osculating planes comprise two cubic envelopes lying on a certain
Cayley surface of the dual space. With respect to this dual Cayley surface the
given cubic envelopes correspond to the parameters
\begin{equation}
  \alpha'=-\frac23\,\alpha,\;
  \alphaq\,'=-\frac23\,\alphaq,\;
   \beta'=\betaq\,'=\frac32, \mbox{ and } \gamma'= \gammaq\,'= 0,
\end{equation}
respectively. This means contact of order four for the cubic envelopes. In this
way it is possible to link the results in (c) and (f) from Theorem \ref{thm:1}.
This correspondence can be established in a more general form
\cite[Theorem~4]{havl-04z}.

\par

Acknowledgement. The figure of the Cayley surface was created with the help of
a macro package (\emph{Maple goes POV-Ray}\/) written by Matthias Wielach in
one of the author's laboratory courses at the Vienna University of Technology.

%%\bibliographystyle{plain}
%%\bibliography{d:/forschung/separata/ketten}

{\small
%-------------------------------------------
% Bitte in die nachstehenden Leerzeilen die
% ADRESSE DES AUTORS eintragen /
% Please enter the AUTHOR'S detailed ADDRESS
% in the following blank lines
% z.B. / e.g.
% Max Mustermann\\
% Institut f"ur Geometrie\\
% Technische Universit"at Dresden\\
% D-01062 Dresden
%-------------------------------------------

Hans Havlicek\\
Forschungsgruppe Differentialgeometrie und Geometrische Strukturen\\
Institut f\"ur Diskrete Mathematik und Geometrie\\
Technische Universit\"at\\
Wiedner Hauptstra{\ss}e 8--10/104\\
A-1040 Wien\\
Austria\\
Email: havlicek@geometrie.tuwien.ac.at}
\end{document}